\documentclass[12pt,a4paper]{amsart} 

\usepackage{latexsym}
\usepackage{amsmath}
\usepackage{amsthm}
\usepackage{latexsym}
\usepackage{amssymb}
\usepackage{subcaption}
\usepackage{caption}
\usepackage[a4paper , lmargin = {3cm} , rmargin = {3cm} , tmargin = {2.5cm} , bmargin = {2.5cm} ]{geometry}

\usepackage{hyperref}
\usepackage{tikz}

\usetikzlibrary{graphs}
\usetikzlibrary{graphs.standard}
\theoremstyle{plain}
\newtheorem*{NewTheoremA}{Theorem A}
\newtheorem*{NewPropositionB}{Proposition B}
\newtheorem*{NewCorollaryC}{Corollary C}

\newtheorem{theorem}{Theorem}[section]
\newtheorem{lemma}[theorem]{Lemma}

\newtheorem{example}[theorem]{Example}
\newtheorem{proposition}[theorem]{Proposition}
\theoremstyle{definition}
\newtheorem{definition}[theorem]{Definition}
\usepackage{graphicx}

\title[]{On number of ends of graph products of groups}
\author{Olga Varghese}

\date{\today}

\address{Olga Varghese\\
Department of Mathematics\\
M\"unster University\\ 
Einsteinstra\ss e 62\\
48149 M\"unster (Germany)}
\email{olga.varghese@uni-muenster.de}

\begin{document}
\pagenumbering{arabic}
\begin{abstract}

Given a finite simplicial graph $\Gamma=(V,E)$ with a vertex-labelling $\varphi:V\rightarrow\left\{\text{non-trivial finitely generated groups}\right\}$, the graph product $G_\Gamma$ is the free product of the vertex groups $\varphi(v)$ with added relations that imply elements of adjacent vertex groups commute. For a quasi-isometric invariant $\mathcal{P}$, we are interested in understanding under which combinatorial conditions on the graph $\Gamma$ the graph product $G_\Gamma$ has  property $\mathcal{P}$. In this article our emphasis  is on number of ends of a graph product $G_\Gamma$. In particular, we   obtain a complete characterization of number of ends of a graph product of finitely generated groups. 
\end{abstract}

\thanks{{Funded by the Deutsche 
Forschungsgemeinschaft (DFG, German Research Foundation) under Germany's 
Excellence Strategy EXC 2044--390685587, Mathematics M\"unster: Dynamics-Geometry-Structure.}}
\maketitle

\section{Introduction}
As the main protagonists in this article are graph products of groups $G_\Gamma$, we start with the definition of these objects.
Given a finite simplicial graph $\Gamma=(V, E)$ with a vertex-labelling 
$\varphi:V\rightarrow\left\{\text{non-trivial finitely generated groups}\right\}$, the \emph{graph product of groups} $G_\Gamma$ is defined as the quotient
\[
\left( \underset{v\in V}{\ast} \varphi(v) \right) / \langle \langle [\varphi(v),\varphi(w)]=1, \{u,w\} \in E \rangle \rangle.
\]
Before we proceed with the introduction we discuss some examples of graph products of groups. Let us consider the following vertex-labelled graphs:
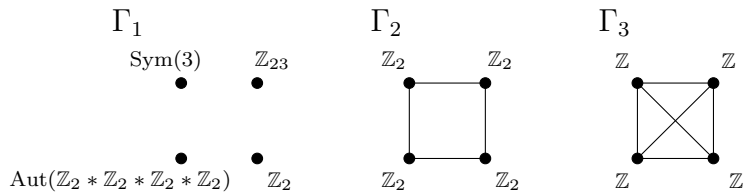
\begin{figure}[h]
\begin{center}
\begin{tikzpicture}
\node at (-0.7,1.8) {$\Gamma_1$};
\draw[fill=black]  (0,0) circle (2pt);
\node at (-0.8,-0.3) {\tiny{${\rm Aut}(\mathbb{Z}_2*\mathbb{Z}_2*\mathbb{Z}_2*\mathbb{Z}_2)$}};
\draw[fill=black]  (1,0) circle (2pt);
\node at (1.3,-0.3) {\tiny{$\mathbb{Z}_2$}};
\draw[fill=black]  (1,1) circle (2pt);
\node at (1.2,1.3) {\tiny{$\mathbb{Z}_{23}$}};
\draw[fill=black]  (0,1) circle (2pt);
\node at (-0.2,1.3) {\tiny{${\rm Sym}(3)$}};

\node at (2.7,1.8) {$\Gamma_2$};
\draw[fill=black]  (3,0) circle (2pt);
\node at (2.8,-0.3) {\tiny{$\mathbb{Z}_2$}};
\draw[fill=black]  (4,0) circle (2pt);
\node at (4.3,-0.3) {\tiny{$\mathbb{Z}_2$}};
\draw[fill=black]  (4,1) circle (2pt);
\node at (4.2,1.3) {\tiny{$\mathbb{Z}_{2}$}};
\draw[fill=black]  (3,1) circle (2pt);
\node at (2.8,1.3) {\tiny{$\mathbb{Z}_2$}};
\draw (3,0)--(4,0);
\draw (4,0)--(4,1);
\draw (4,1)--(3,1);
\draw (3,1)--(3,0);

\node at (5.7,1.8) {$\Gamma_3$};
\draw[fill=black]  (6,0) circle (2pt);
\node at (5.8,-0.3) {\tiny{$\mathbb{Z}$}};
\draw[fill=black]  (7,0) circle (2pt);
\node at (7.3,-0.3) {\tiny{$\mathbb{Z}$}};
\draw[fill=black]  (7,1) circle (2pt);
\node at (7.2,1.3) {\tiny{$\mathbb{Z}$}};
\draw[fill=black]  (6,1) circle (2pt);
\node at (5.8,1.3) {\tiny{$\mathbb{Z}$}};

\draw (6,0)--(7,0);
\draw (6,0)--(6,1);
\draw (6,0)--(7,1);
\draw (7,0)--(7,1);
\draw (6,1)--(7,1);
\draw (7,0)--(6,1);

\end{tikzpicture}
\end{center}
\end{figure}

The graph $\Gamma_1$ has no edges, hence $G_{\Gamma_1}={\rm Sym}(3)*\mathbb{Z}_{23}*\mathbb{Z}_2*{\rm Aut}(\mathbb{Z}_2*\mathbb{Z}_2*\mathbb{Z}_2*\mathbb{Z}_2)$. The graph product $G_{\Gamma_2}=(\mathbb{Z}_2*\mathbb{Z}_2)\times(\mathbb{Z}_2*\mathbb{Z}_2)$ and $G_{\Gamma_3}=\mathbb{Z}\times\mathbb{Z}\times\mathbb{Z}\times\mathbb{Z}$. Important cases of graph products of groups arise when we specify the vertex groups.
If all vertex groups are cyclic of order two, then $G_\Gamma$ is a right-angled Coxeter group and if all vertex groups are infinite cyclic, then $G_\Gamma$ is a right-angled Artin group. 

Let us observe one more example.
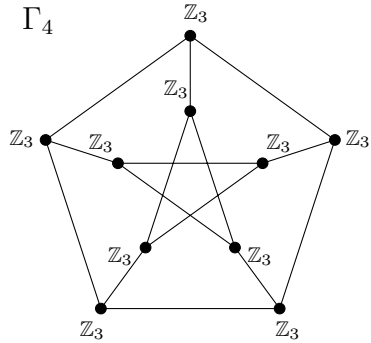
\begin{tiny}
\begin{center}
\begin{tikzpicture}
\draw (18:2cm) -- (90:2cm) -- (162:2cm) -- (234:2cm) --
(306:2cm) -- cycle;
\draw (18:1cm) -- (162:1cm) -- (306:1cm) -- (90:1cm) --
(234:1cm) -- cycle;

\foreach \x in {18,90,162,234,306}{
\draw (\x:1cm) -- (\x:2cm);
\draw[fill=black] (\x:2cm) circle (2pt);
\draw[fill=black] (\x:1cm) circle (2pt);

}
\node at (16:2.3cm) {\tiny{$\mathbb{Z}_3$}};
\node at (88:2.3cm) {\tiny{$\mathbb{Z}_3$}};
\node at (164:2.3cm) {\tiny{$\mathbb{Z}_3$}};
\node at (236:2.3cm) {\tiny{$\mathbb{Z}_3$}};
\node at (304:2.3cm) {\tiny{$\mathbb{Z}_3$}};

\node at (26:1.3cm) {\tiny{$\mathbb{Z}_3$}};
\node at (98:1.3cm) {\tiny{$\mathbb{Z}_3$}};
\node at (154:1.3cm) {\tiny{$\mathbb{Z}_3$}};
\node at (225:1.3cm) {\tiny{$\mathbb{Z}_3$}};
\node at (315:1.3cm) {\tiny{$\mathbb{Z}_3$}};

\node at (-2,2.2) {\normalsize{$\Gamma_4$}};
\end{tikzpicture}
\end{center}
\end{tiny}
The corresponding graph product $G_{\Gamma_4}$ does not admit a description in terms of free and direct products, but it can be written as a non-trivial  amalgamated product of graph products of groups $G_{\Delta_1}*_{G_{\Delta_2}} G_{\Delta_3}$, where $\Delta_1, \Delta_2$ and $\Delta_3$ are full subgraphs of $\Gamma_4$.

An interesting and challenging question is how graph properties of $\Gamma$ influence the algebraic and geometric structure of $G_{\Gamma}$. Many group theoretical properties of graph products of groups have been translated in combinatorial structure of the defining graph $\Gamma$. For example, Gromov hyperbolicity, virtual freeness, planarity of the Cayley graph with respect to the canonical generating set \cite[Theorem A]{V} and coherence \cite[Theorem A]{VVV}. Let us focus on the first two properties and give translations of these geometric structures in the combinatorial world of graphs. 
\vspace{0.15cm}

\begin{center}
{\bf Graph products of finite vertex groups dictionary}
\end{center}
\begin{center}
\begin{tabular}{|l|l|l|}
\hline
\multicolumn{1}{|c|}{\emph{Property of $G_\Gamma$}} &\multicolumn{1}{c|}{\emph{Structure of $\Gamma$}}& \multicolumn{1}{c|}{\emph{Reference}}\\
\hline 
hyperbolic & $\Gamma$ contains no induced cycle of length $4$ & \cite[Main Theorem]{Meier}\\
\hline 
virtually free  &  $\Gamma$ contains no induced cycle of length $\geq 4$  & \cite[Theorem 1.1]{LS} \\
\hline 
 \end{tabular}
\end{center}
\hspace{1cm}

Hence, hyperbolicity and virtual freeness of $G_\Gamma$ can be checked easily on the graph $\Gamma$. The above graph $\Gamma_4$ is the Petersen graph with a vertex-labelling. This graph has no induced cycle of length $4$, but it has an induced cycle of length $5$. Therefore the group $G_{\Gamma_4}$ is hyperbolic, but not virtually free. 

Gromov hyperbolicity and virtual freeness are quasi-isometric invariants. 
Another example of a quasi-isometric invariant is the number of ends. For a finitely generated group $G$ the \emph{number of ends}, denoted by $e(G)$, is defined as follows: 
\[
e(G):=\sup\left\{ \#\pi^{u}_0({\rm Cay}(G, S)-C)\mid C\subseteq{\rm Cay}(G, S)\text{ finite subgraph}\right\},
\]
where ${\rm Cay}(G, S)$ is the Cayley graph for $G$ with respect to a finite generating set $S$ and $\#\pi^{u}_0({\rm Cay}(G,$  $S)-C)$ is the number of unbounded connected components of ${\rm Cay}(G, S)-C$. It was proven by Hopf in \cite{Hopf} that a finitely generated group has $0, 1, 2$ or $\infty$ ends.  Given two finitely generated infinite groups $G$ and $H$, the number of ends of the direct product $G\times H$ is $1$ and of the free product $G*H$ is $\infty$. Graph products of groups interpolate between direct and free products, hence it is natural to ask how the shape of the graph $\Gamma$ determines the number of ends of the graph product of groups $G_\Gamma$. We prove that the number of ends of a graph product of groups $G_\Gamma$ is visible in the defining graph $\Gamma$.

\begin{NewTheoremA}
Let $G_\Gamma$ be a graph product of non-trivial finitely generated groups. The graph product $G_\Gamma$ has more than one end if and only if either:
\begin{enumerate}
\item[(i)] The graph $\Gamma$ is complete with exactly one vertex group which has more than one end and all others vertex groups are finite, or
\item[(ii)] there are non-empty full subgraphs $\Gamma_1$ and $\Gamma_2$ of $\Gamma$  (neither containing the other) such that $G_\Gamma$ decomposes as the (non-trivial) amalgamated product 
\[
G_{\Gamma}\cong G_{\Gamma_1}*_{G_{\Gamma_1\cap\Gamma_2}} G_{\Gamma_2}
\]
where $G_{\Gamma_1\cap\Gamma_2}$ is a finite group. 
\end{enumerate}
\end{NewTheoremA}

A simplicial graph $\Gamma=(V,E)$ is said to have a \emph{separating subgraph} $\Delta=(X,Y)$ if  the induced subgraph of $\Gamma$ with the vertex set $V-X$ has at least two connected components. Note that, if the graph $\Gamma$ is disconnected then the empty graph is a separating subgraph. Theorem A implies that in order to check  if a non-complete graph product of groups has more than one end or not, one only need check $\Gamma$ for complete separating subgraphs each of whose vertex groups are finite.
\begin{NewPropositionB}
Let $G_\Gamma$ be a graph product of non-trivial finitely generated groups. Decompose $\Gamma$ as a join $\Gamma_1*\Gamma_2$ where $\Gamma_1$ is a subgraph generated by the vertices which are adjacent to all the vertices of $\Gamma$. The graph product $G_\Gamma$ is $2$-ended if and only if either:
\begin{enumerate}
\item[(i)] The graph $\Gamma$ is complete with exactly one $2$-ended vertex group and all other vertex groups being finite, or
\item[(ii)] the subgraph $\Gamma_1$ is complete with finite vertex groups and $\Gamma_2=(\left\{\mathbb{Z}_2, \mathbb{Z}_2\right\},\emptyset)$.
\end{enumerate} 
\end{NewPropositionB}

We immediately obtain the following description of number of ends of graph products of finite groups.
\begin{NewCorollaryC}
Let $G_\Gamma$ be a graph product of non-trivial finite groups. Decompose $\Gamma$ as a join $\Gamma_1*\Gamma_2$ where $\Gamma_1$ is a subgraph generated by the vertices which are adjacent to all the vertices of $\Gamma$.
\begin{enumerate}
\item[(i)] $e(G_\Gamma)=0$ if and only if the graph $\Gamma$ is complete.
\item[(ii)] $e(G_\Gamma)=1$ if and only if $\Gamma$ is connected, not complete and has no complete separating subgraph.
\item [(iii)] $e(G_\Gamma)=2$ if and only if  $\Gamma_2$ is isomorphic to $(\left\{\mathbb{Z}_2, \mathbb{Z}_2\right\},\emptyset)$. 
\item [(iv)] $e(G_\Gamma)=\infty$ if and only if $\Gamma$ has a complete separating subgraph $\Delta$ and $\Gamma_2$ is not isomorphic to $(\left\{\mathbb{Z}_2, \mathbb{Z}_2\right\},\emptyset)$.
\end{enumerate}
\end{NewCorollaryC}

The number of ends of arbitrary Coxeter groups were characterized in terms of the defining graph in \cite[Cor.16, Cor.17]{MT} and in \cite[Chapter 8.7]{Davis}.  These results were the starting point for our investigations on number of ends of graph products of groups. Our proofs of Theorem A and Proposition B involve a structure theorem of groups with more than one end by Bergman \cite{Bergman} and Stallings \cite{St} and Bass-Serre theory of group actions on simplicial trees \cite{Serre}.

\subsection*{Acknowledgment}
The author thanks the referee for careful reading of the manuscript and for giving many constructive comments which helped improving the quality of the manuscript. The author would like also to thank Anthony Genevois for his comments concerning Theorem A which helped to prove this theorem in more general context.

\section{Graph products of groups}
We start by reviewing the concept of simplicial graphs. A {\it simplicial graph} $\Gamma=(V,E)$ consists of a set $V$ and a set $E$ of $2$-element subsets of $V$. The elements of $V$ are called {\it vertices} and the elements of $E$ are its {\it edges}.  If $V'\subseteq V$ and $E'\subseteq E$, then $\Gamma'=(V', E')$ is called a {\it subgraph} of $\Gamma$. If $\Gamma'$ is a subgraph of $\Gamma$ and $E'$ contains all the edges $\left\{v, w\right\}\in E$ with $v, w\in V'$, then $\Gamma'$ is called a {\it full} or \emph{induced} subgraph of $\Gamma$. A \emph{path} is a graph $P=(V,E)$ of the form 
\[
V=\left\{v_0,\ldots,v_n\right\}\text{ and }E=\left\{\left\{v_0,v_1\right\},\left\{v_1,v_2\right\},\ldots ,\left\{v_{n-1},v_n\right\}\right\}
\]
 where $v_i$ are all distinct.  A graph $\Gamma=(V,E)$ is called \emph{connected} if any two vertices $v,w\in V$ are contained in a subgraph $\Gamma'$ of $\Gamma$ such that $\Gamma'$ is a path. A maximal connected subgraph of $\Gamma$ is called a \emph{connected component of $\Gamma$}. A graph is called \emph{complete} if there is an edge for every pair of distinct vertices. A non-empty simplicial graph $\Gamma$ is called a \emph{tree} if $\Gamma$ is connected and has no induced cycle of length $\geq 3$.  A graph $\Gamma=(V,E)$ has a  \emph{separating subgraph} $\Delta=(X,Y)$ if the induced subgraph of $\Gamma$ with the vertex set $V-X$ is not connected. If $\Gamma$ is disconnected, then the empty graph is a separating subgraph. 

As the main objects in this article are graph products of groups, we
proceed with the definition of these groups, and establish some notation to be used throughout. Given a simplicial graph $\Gamma=(V,E)$, a \emph{vertex-labelling} of $\Gamma$ is a map $\varphi:V \rightarrow\left\{\text{non-trivial finitely generated groups}\right\}$. For $v\in V$ we denote by $G_v$ the image of $v$ under $\varphi$. 
\begin{definition}
Let $\Gamma=(V, E)$ be a finite simplicial vertex-labelled graph. The \emph{graph product of groups} $G_\Gamma$ is defined as the quotient
\[
\left( \underset{v\in V}{\ast} G_v \right) / \langle \langle [G_u,G_w]=1, \{u,w\} \in E \rangle \rangle.
\]
\end{definition}
Graph products of infinite cyclic vertex groups were introduced by Baudisch in \cite{B} and of arbitrary vertex groups by Green in \cite{G}.
Two cases to keep in mind are that the graph product $G_\Gamma$ is the free product of the vertex groups if $\Gamma$ does not contain any edge, and it is the direct product of the vertex groups if $\Gamma$ is a complete graph. 

Let $\mathcal{P}$ be an interesting property of groups. We are interested in understanding under which combinatorial conditions on the graph $\Gamma$ the group $G_\Gamma$ has property $\mathcal{P}$. A basic algebraic property of groups is finiteness. Finiteness of $G_\Gamma$ can be described easily in terms of $\Gamma$. Before we prove a characterization of finiteness of $G_\Gamma$ let us mentioned two important properties of graph products of groups  which will be needed several times in this article.

\begin{lemma}(\cite[Lemma 3.20]{G})
\label{fullsubgroup}
Let $\Gamma=(V,E)$ be a finite simplicial vertex-labelled graph and $\Delta$ be a subgraph. 
\begin{enumerate}
\item[(i)] If $\Delta$ is a full subgraph, then the subgroup of $G_\Gamma$ generated by $G_v, v\in \Delta$ is canonically isomorphic to the graph product $G_\Delta$ and is called a \emph{special subgroup} of $G_\Gamma$. 
In particular, each vertex group $G_v, v\in V$ is a special subgroup.
\item[(ii)] If $\Gamma$ has full subgraphs $\Gamma_1, \Gamma_2$ with $\Gamma=\Gamma_1\cup\Gamma_2$, then $G_\Gamma$ is an amalgamated product $G_{\Gamma_1}*_{G_{\Gamma_1\cap\Gamma_2}}G_{\Gamma_2}$.
\end{enumerate}
\end{lemma}

We immediately obtain an answer to the question under which condition on the graph $\Gamma$ the graph product of groups $G_\Gamma$ is finite.
\begin{proposition}
\label{finite}
Let $\Gamma=(V,E)$ be a finite simplicial vertex-labelled graph and $G_\Gamma$ be the corresponding graph product of groups. The group $G_\Gamma$ is finite if and only if $\Gamma$ is complete and $G_v$ is finite for all $v\in V$. 
\end{proposition}
\begin{proof}
If $\Gamma$ is complete and $G_v$ is finite for all $v\in V$, then $G_\Gamma$ is a direct product of  finite groups and is therefore finite. For the other direction let us assume that $\Gamma$ is not complete or at least one vertex group is infinite. If $\Gamma$ is not complete then there exist two vertices $v$ and $w$ which are not connected by an edge. The group $G_v*G_w$ is by Lemma \ref{fullsubgroup} a subgroup of $G_\Gamma$, hence $G_\Gamma$ is infinite. If $\Gamma$ has at least one infinite vertex group $G_x$, then the group $G_x$ is by Lemma \ref{fullsubgroup} a subgroup of $G_\Gamma$, thus $G_\Gamma$ is again infinite.
\end{proof}

\section{Quasi-isometric invariants}
In geometric group theory there is a canonical way to consider an abstract group as a metric space, namely by replacing the group with its Cayley graph.  
Let $G$ be a finitely generated group and $S\subseteq G-\left\{1\right\}$ be a finite generating set of $G$. The \emph{Cayley graph for $G$ with respect to $S$}, denoted by ${\rm Cay}(G,S)$, is a directed edge-labbeled graph given as follows: the vertex set $V=G$, the edge set $E=\left\{(g,gs)\mid g\in G, s\in S\right\}$, where for a directed edge $(g,gs)$ the vertex $g$ is the initial vertex and $gs$ is the terminal vertex. The edge-labelling is given as follows $\varphi: E\rightarrow S$, $(g,gs)\mapsto s$. The group $G$ acts on ${\rm Cay}(G,S)$ via left-multiplication and one can show that this map is an isomorphism between $G$ and the automorphism group of ${\rm Cay}(G,S)$. We consider ${\rm Cay}(G,S)$ as a metric space with the path metric. Clearly, the Cayley graph ${\rm Cay}(G,S)$ depends on the choice of the generating set of the group. For example ${\rm Cay}(\mathbb{Z}_5,\left\{1\right\})$ and ${\rm Cay}(\mathbb{Z}_5,\left\{1,2,3,4\right\})$ are very different from the graph theoretical point of view, but these graphs are quasi-isometric.

\begin{definition}
Let $(X, d_X)$ and $(Y,d_Y)$ be metric spaces. 
\begin{enumerate}
\item[(i)] A map $f: X\rightarrow Y$ is called a \emph{quasi-isometric embedding} if there exist $K,C\in\mathbb{R}, K\geq 0, C\geq 0$ such that
\[
\frac{1}{K}d_X(x_1,x_2)-C\leq d_Y(f(x_1),f(x_2))\leq Kd_X(x_1, x_2)+C
\]
for all $x_1, x_2\in X$.
\item[(ii)] A quasi-isometric embedding $f:X\rightarrow Y$ is called a \emph{quasi-isometry}, if there exists a constant $D>0$ such that for every $y\in Y$ there exist an $x\in X$ so that $d_Y(f(x), y)\leq D$.
\item[(iii)] The metric spaces $(X, d_X)$ and $(Y,d_Y)$ are \emph{quasi-isometric} if there exists a quasi-isometry $f:X\rightarrow Y$.
\end{enumerate}
\end{definition}

Given a finitely generated group $G$ and two finite generating sets $S_1$ and $S_2$, the Cayley graphs ${\rm Cay}(G,S_1)$ and ${\rm Cay}(G,S_2)$ are quasi-isometric, see \cite[I 8.17]{BH}. Two finitely generated groups $G$ and $H$ are called \emph{quasi-isometric} if their Cayley graphs are quasi-isometric. One important fact is that every finite index subgroup $H$ of a finitely generated group $G$ is quasi-isometric to $G$. 

A big goal in geometric group theory, which was initiated by Gromov, is a classification of finitely generated groups up to quasi-isometry. Quasi-isometric invariants allow us to distinguish between quasi-isometry classes of groups. A property $\mathcal{P}$ of a group is called a \emph{quasi-isometric invariant} if for every two finitely generated groups $G_1$ and $G_2$ which are quasi-isometric, $G_1$ has property $\mathcal{P}$ if and only if $G_2$ has property $\mathcal{P}$. It is obvious, that to be finite is a quasi-isometric invariant.  Further examples of quasi-isometric invariants include Gromov hyperbolicity, virtual freeness and the number of ends.  

\begin{definition}
Let $G$ be a finitely generated group, $S$ be a finite generating set of $G$ and  ${\rm Cay}(G, S)$ be the corresponding Cayley graph. The \emph{number of ends}, denoted by $e(G)$, is defined as follows: 
\[
e(G):=\sup\left\{ \#\pi^{u}_0({\rm Cay}(G, S)-C)\mid C\subseteq{\rm Cay}(G, S)\text{ finite subgraph}\right\},
\]
where  $\#\pi^{u}_0({\rm Cay}(G, S)-C)$ is the number of unbounded connected components of ${\rm Cay}(G, S)-C$.
\end{definition}

The number of ends does not depend on a finite generating set of $G$. Furthermore the number of ends is a quasi-isometric invariant \cite[Part I Proposition 8.29]{BH}. For example, $e(\mathbb{Z})=2$ and $e(F_n)=\infty$ where $F_n$ is a free group of rank $n\geq 2$.

A proof of the following theorem can be found for instance in \cite[Part I Theorem 8.32]{BH}, \cite{Bergman}, \cite{Hopf}, \cite{St}.
\begin{theorem}
\label{NumberOfEnds}
Let $G$ be a finitely generated group.
\begin{enumerate}
\item[(i)] The group $G$ has $0, 1, 2$ or $\infty$ many ends.
\item[(ii)] The group $G$ has $0$ ends if and only if it is finite.
\item[(iii)] The group $G$ has $2$ ends if and only if it is virtually $\mathbb{Z}$ (i.e. $G$ has a subgroup of finite index isomorphic to $\mathbb{Z}$).
\item[(iv)] The group $G$ has more than $1$ end if and only if $G$ is a non-trivial amalgamated product $A*_C B$ or ${\rm HNN}$-extension $A*_C$ with $C$ finite.
\item[(v)] The group $G$ has infinitely many ends if and only if $G$ is a non-trivial amalgamated product $A*_C B$ or ${\rm HNN}$-extension $A*_C$ with $C$ finite and $|A:C|\geq 3$,  $|B:C|\geq 2$.
\item[(vi)] Suppose $G=A*_C B$ where $A$ and $B$ are $1$-ended and $C$ is infinite and finitely generated, then $G$ is $1$-ended.
\end{enumerate}
\end{theorem}

In particular, it follows that the number of ends of the free product of two infinite groups $H$ and $K$ is $\infty$ and it is not hard to prove that the number of ends of the direct product of two infinite groups $H$ and $K$ is $1$. Let us mention that using Theorem \ref{NumberOfEnds} we can not compute the number of ends of the graph product which is defined via the graph $\Gamma_5$. 
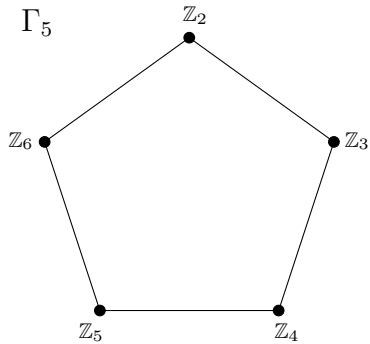
\begin{tiny}
\begin{center}
\begin{tikzpicture}
\draw (18:2cm) -- (90:2cm) -- (162:2cm) -- (234:2cm) --
(306:2cm) -- cycle;

\foreach \x in {18,90,162,234,306}{
\draw[fill=black] (\x:2cm) circle (2pt);
}
\node at (16:2.3cm) {\tiny{$\mathbb{Z}_3$}};
\node at (88:2.3cm) {\tiny{$\mathbb{Z}_2$}};
\node at (164:2.3cm) {\tiny{$\mathbb{Z}_6$}};
\node at (236:2.3cm) {\tiny{$\mathbb{Z}_5$}};
\node at (304:2.3cm) {\tiny{$\mathbb{Z}_4$}};

\node at (-2,2.2) {\normalsize{$\Gamma_5$}};
\end{tikzpicture}
\end{center}
\end{tiny}
The graph $\Gamma_5$ is connected, not complete and has no complete separating subgraph, thus the group $G_{\Gamma_5}$ has one end by Corollary C. Let $\Gamma_6$ be a hexagon with vertices $v_1, \ldots, v_6$. For $i$ even, let $G_{v_i}$ be finite and for $i$ odd, let $G_{v_i}$ be infinite ended. Then again Theorem 3.3 doesn't help, but using Theorem A we obtain $e(G_{\Gamma_6})=1$.

\begin{lemma}
\label{freeProd}
Let $G$ and $H$ be finite groups. The free product $G*H$ has $2$ ends if and only if $G\cong H\cong\mathbb{Z}_2$.
\end{lemma}
\begin{proof}
If $G\cong H\cong\mathbb{Z}_2$, then the free product $G*H$ is isomorphic to the infinite dihedral group $D_\infty$ which is virtually infinite cyclic. Hence $G*H$ is quasi-isometric to $\mathbb{Z}$ and $\mathbb{Z}$ has $2$ ends.

Assume for contradiction that $\#G\geq 3$. The kernel of the canonical map $\pi:G*H\rightarrow G\times H$ is a free group of rank $\geq 2$, see \cite[I \S 1 Prop. 4]{Serre}. We obtain $e(G*H)=e({\rm ker}(\pi))=\infty$, a contradiction. 
\end{proof}

A description of virtually infinite cyclic right-angled Coxeter groups in terms of the defining graph is proven in \cite[Cor.17]{MT} and in \cite[Chapter 8.7]{Davis}. Thanks to the referee we give here a characterization of virtually infinite cyclic graph products of arbitrary vertex groups.

\begin{NewPropositionB}
Let $G_\Gamma$ be a graph product. Decompose $\Gamma$ as a join $\Gamma_1*\Gamma_2$ where $\Gamma_1$ is a subgraph generated by the vertices which are adjacent to all the vertices of $\Gamma$. The graph product $G_\Gamma$ is $2$-ended if and only if either:
\begin{enumerate}
\item[(i)] The graph $\Gamma$ is complete with exactly one $2$-ended vertex group and all other vertex groups being finite, or
\item[(ii)] the subgraph $\Gamma_1$ is complete with finite vertex groups and $\Gamma_2=(\left\{\mathbb{Z}_2, \mathbb{Z}_2\right\},\emptyset)$.
\end{enumerate} 
\end{NewPropositionB}
\begin{proof}
We first prove that any infinite subgroup of a virtually infinite cyclic group is virtually infinite cyclic.
Let $H$ be a virtually infinite cyclic group and $K\subseteq H$ an infinite subgroup. Every virtually infinite cyclic group contains a normal infinite cyclic group of finite index. Thus, the group $H$ has a normal subgroup $L$ of finite index with $L\cong\mathbb{Z}$. Let $f:K\rightarrow H/L$ be the quotient map. We obtain $K/{\rm ker}(f)\cong f(K)$. The image of $K$ under $f$ is a finite group and ${\rm ker}(f)$ is an infinite subgroup of $L$, hence ${\rm ker}(f)$ is isomorphic to $\mathbb{Z}$, this fact proves that $K$ is virtually infinite cyclic.

Clearly either (i) or (ii) imply $G_\Gamma$ is $2$-ended. More precisely, if $\Gamma$ is complete with a one $2$-ended vertex group $G_v$ and all other vertex groups being finite, then $G_\Gamma$ is a direct product of the vertex groups and $G_\Gamma$ is quasi-isometric to $G_v$. If the subgraph $\Gamma_1$ is complete with finite vertex groups and $\Gamma_2=(\left\{\mathbb{Z}_2, \mathbb{Z}_2\right\},\emptyset)$, then $G_\Gamma$ is isomorphic to $G_{\Gamma_1}\times (\mathbb{Z}_2*\mathbb{Z}_2)$. The group $G_{\Gamma_1}$ is a direct product of finite groups and is therefore finite. Hence, the group $G_\Gamma$ is quasi-isometric to $\mathbb{Z}_2*\mathbb{Z}_2$ and this group is virtually infinite cyclic.

To prove the other direction assume that (i) does not hold. We prove that (ii) does hold. We already know that all vertex groups of $\Gamma$ are finite or $2$-ended. If $\Gamma$ is complete and (i) does not hold, then at least two vertex groups are $2$-ended. By Lemma \ref{fullsubgroup}(i) the graph product $G_\Gamma$ contains the direct product of these two groups as a $1$-ended subgroup, which is impossible. Hence $\Gamma$ is not complete. We choose vertices $v, w$ which are not connected by an edge. Then the subgroup generated by $G_v$ and $G_w$ is isomorphic to $G_v*G_w$ and this subgroup has infinitely many ends unless $G_v\cong G_w\cong \mathbb{Z}_2$, see Lemma \ref{freeProd}. Thus $G_v\cong G_w\cong\mathbb{Z}_2$. Let $u$ be a vertex of $\Gamma-\left\{v,w\right\}$. If $u$ and $v$ are not connected by an edge, then the subgroup generated by $G_v, G_w, G_u$ is isomorphic to $G_v*G_{(\left\{u,w\right\},\emptyset)}$ or to $G_v*G_{(\left\{u,w\right\},\left\{\left\{u,w\right\}\right\})}$. Both groups are infinite ended, which is impossible. Instead there is an edge between $u$ and $v$. Similarly there is an edge between $u$ and $w$. Hence $G_\Gamma\cong(G_v*G_w)\times H$ where $H$ is the graph product of the full subgraph on $V-\left\{v,w\right\}$. Since $G_\Gamma$ is not $1$-ended, $H$ is finite. This finishes the proof.  
\end{proof}

\section{Trees of groups}
The aim of this section is to bring together two concepts: graph products of groups and trees of groups. The standard reference for trees of groups is \cite{Serre}.

\begin{definition}
Let $\Delta=(X,Y)$ be a finite tree with vertex set $X$ and edge set $Y$. Let $\Psi:X\rightarrow\left\{\text{groups}\right\}$ be a vertex-labelling of $\Delta$ and $\Phi: Y\rightarrow\left\{\text{groups}\right\}$ be an edge-labelling of $\Delta$. For every edge $\left\{x_1, x_2\right\}\in Y$ we have monomorphisms $f_{\left\{x_1, x_2\right\}, x_1}:\Phi(\left\{x_1, x_2\right\})\hookrightarrow\Psi(x_1)$ and $f_{\left\{x_1,x_2\right\}, x_2}:\Phi(\left\{x_1, x_2\right\})\hookrightarrow\Psi(x_2)$.

The \emph{tree of groups $\pi_1(\Delta)$} is defined as the quotient 
\[
\left( \underset{x\in X}{\ast} \Psi(x) \right) / \langle \langle f_{\left\{x_1, x_2\right\}, x_1}(a)=f_{\left\{x_1,x_2\right\},x_2}(a), a\in\Phi(\left\{x_1, x_2\right\}), \left\{x_1,x_2\right\}\in Y \rangle \rangle.
\]
\end{definition}
Let $G_\Gamma$ be a graph product. For our purpose, we will consider trees of groups where the vertex and the edge groups are special subgroups of $G_\Gamma$ and the monomorphisms from the edge-groups into the vertex groups are canonical embeddings.

\begin{example}
Let us observe the following two graphs.
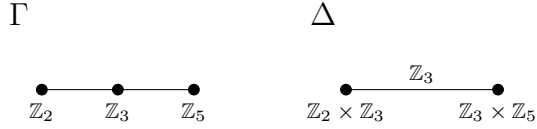
\begin{figure}[h]
\begin{center}
\begin{tikzpicture}
\draw[fill=black]  (0,0) circle (2pt);
\node at (0,-0.3) {\tiny{$\mathbb{Z}_2$}};
\draw[fill=black]  (1,0) circle (2pt);
\node at (1,-0.3) {\tiny{$\mathbb{Z}_3$}};
\draw[fill=black]  (2,0) circle (2pt);
\node at (2,-0.3) {\tiny{$\mathbb{Z}_5$}};
\draw (0,0)--(1,0);
\draw (1,0)--(2,0);

\draw[fill=black]  (4,0) circle (2pt);
\node at (4,-0.3) {\tiny{$\mathbb{Z}_2\times\mathbb{Z}_3$}};
\draw[fill=black]  (6,0) circle (2pt);
\node at (6,-0.3) {\tiny{$\mathbb{Z}_3\times\mathbb{Z}_5$}};
\draw (4,0)--(6,0);
\node at (5,0.2) {\tiny{$\mathbb{Z}_3$}};

\node at (-0.3,1) {$\Gamma$};
\node at (3.7,1) {$\Delta$};
\end{tikzpicture}
\end{center}
\end{figure}

The graph product $G_\Gamma$ is isomorphic to  $(\mathbb{Z}_2\times\mathbb{Z}_3)*_{\mathbb{Z}_3}(\mathbb{Z}_3\times\mathbb{Z}_5)$. Clearly, the graph product $G_\Gamma$ is isomorphic to the tree of groups $\pi_1(\Delta)$.
\end{example}

Now we are ready to prove Theorem A. For arbitrary Coxeter groups similar result is proven in \cite[Theorem 1]{MT} and we follow the main ideas of that proof.

\begin{NewTheoremA}
Let $G_\Gamma$ be a graph product of non-trivial finitely generated groups. The graph product $G_\Gamma$ has more than one end if and only if either:
\begin{enumerate}
\item[(i)] The graph $\Gamma$ is complete with exactly one vertex group which has more than one end and all others vertex groups are finite, or
\item[(ii)] there are non-empty full subgraphs $\Gamma_1$ and $\Gamma_2$ of $\Gamma$  (neither containing the other) such that $G_\Gamma$ decomposes as the (non-trivial) amalgamated product 
\[
G_{\Gamma}\cong G_{\Gamma_1}*_{G_{\Gamma_1\cap\Gamma_2}} G_{\Gamma_2}
\]
where $G_{\Gamma_1\cap\Gamma_2}$ is a finite group. 
\end{enumerate}
\end{NewTheoremA}
\begin{proof} 
Clearly either (i) or (ii) imply that $G_\Gamma$ has more than one end.

For the other direction assume first that $\Gamma$ is complete and (i) does not hold. Then $\Gamma$ has at least two vertex groups which are infinite or $\Gamma$ has one $1$-ended vertex group and all other vertex groups are finite or all vertex groups of $\Gamma$ are finite. The group $G_\Gamma$ has by assumption more than one end, hence all above cases are impossible.
 
Now we assume that $\Gamma$ is not complete. The group $G_\Gamma$ is a non-trivial amalgamated product $A*_C B$ or HNN-extension $A*_C$ with $C$ finite, see Theorem 3.3. Thus $G_\Gamma$ acts via left-multiplication on the corresponding Bass-Serre tree $T$. If a subgroup $H\subseteq G_\Gamma$ stabilizes an edge, then $H$ is finite. This follows from the fact that the stabilizer of an edge is the group $gCg^{-1}$ for $g\in G_\Gamma$ which is by assumption finite. Further, if a subgroup $H\subseteq G_\Gamma$ is stabilizing a vertex $gA$ (resp. $gB$), then $H\subseteq gAg^{-1}$ (resp. $H\subseteq gBg^{-1}$). 

If there is a vertex-group $G_v$ which acts on $T$ without a fixed point, then by \cite[I \S 6 Corollary 2]{Serre} there exists an element $g\in G_v$ which acts on $T$ without a fixed point. Further $g$ is a hyperbolic isometry 
and there exists a unique geodesic line $A_g:\mathbb{R}\rightarrow T$ such that
$g(A_g(t))=A_g(t+a)$ for $a\in\mathbb{Z}-\left\{0\right\}$. The set $A_g(\mathbb{R})$ is called the axis of $g$. Note that if an element $h\in G_\Gamma$ commutes with $g$, then it leaves $A_g(\mathbb{R})$ invariant, see \cite[II theorem 7.1(4)]{BH}. Let ${\rm lk}(v)$ be the full subgraph of $\Gamma$ generated by the vertices which are adjacent to $v$. Thus the subgroup generated by $G_{{\rm lk}(v)}$ and $g$ acts on $A_g(\mathbb{R})$ by simplicial isometries. We obtain a map $\langle G_{{\rm lk}(v)}, g \rangle\rightarrow \mathbb{Z}_2*\mathbb{Z}_2$ whose kernel acts on the axis of $g$ trivially and is therefore by assumption finite. Thus the graph product $\langle G_{{\rm lk}(v)}, g \rangle$ is $2$-ended and by Proposition B follows that $G_{{\rm lk}(v)}$ is a finite group. We obtain a non-trivial decomposition
\[
G_\Gamma=G_{V-\left\{v\right\}}*_{G_{{\rm lk}(v)}} G_{{\rm lk}(v)\cup\left\{v\right\}}.
\] 

If all vertex groups act with a fixed point on $T$, then we consider the following set consisting of special subgroups of $G_\Gamma$:
\[
\mathcal{E}=\left\{G_{\Gamma'}\mid \Gamma'=(\left\{v,w\right\},\left\{\left\{v,w\right\}\right\})\subseteq\Gamma \text{ a subgraph}\right\}.
\]
Every group $G_{\Gamma_i}\in\mathcal{E}$ is a direct product of two vertex groups and therefore this group has a fixed point on $T$, see \cite[Corollary 4.4]{FV}. Hence, there exists a vertex $x_i\in T$ which is stabilized by the subgroup $G_{\Gamma_i}$.

We construct a tree of groups $\pi_1(\Delta)$ as follows:
\newline Let $\Delta'$ be the subtree of $T$ generated by the vertices $\left\{x_i\mid G_{\Gamma_i}\in\mathcal{E}\right\}$. So far we have constructed a finite tree. For a tree of groups decomposition we need a vertex and an edge-labelling. For each vertex $x\in\Delta'$ and each edge $e\in\Delta'$ take the maximal special subgroup of $G_\Gamma$ which stabilizes $x$ resp. $e$ and   collapse the edges where the edge group is equal to the corresponding vertex-group. Since $G_\Gamma$ is a non-trivial amalgam (resp. a non-trivial HNN-extension), the (smaller) tree $\Delta=(K,L)$ has at least $2$ vertices.

It is obvious that the edge and the vertex groups of $\Delta$ are special subgroups of $G_\Gamma$ and each edge group is canonically embedded into the corresponding vertex groups. Furthermore, the edge groups are finite subgroups. 

Our goal now is to prove that the tree of groups $\pi_1(\Delta)$ is isomorphic to $G_\Gamma$. We denote by $\Psi:K\rightarrow\left\{\text{special subgroups of }G_\Gamma\right\}$ the vertex-labelling and by $\Phi:L\rightarrow\left\{\text{special subgroups of }G_\Gamma\right\} $ the edge-labelling of $\Delta$.
 
Let $\phi:\pi_1(\Delta)\rightarrow G_\Gamma$ be the extending of the inclusion maps of the vertex groups $\Phi(k)$ into $G_\Gamma$. By construction of $\Delta$, each relator of $G_\Gamma$ is a relator of $\pi_1(\Delta)$. Further, for a vertex-group $G_v\in G_\Gamma$, the vertices and edges of $\Delta$ containing $G_v$ form a subtree. Hence, the map $\phi$ is an isomorphism.

Now we convert the tree of groups $\pi_1(\Delta)$ into an amalgam as follows: 
\newline We take an edge $e\in\Delta$. The graph $(K, L-\left\{e\right\})$ has two non-empty connected components $\Delta_1$ and $\Delta_2$. We obtain $\pi_1(\Delta)\cong\pi_1(\Delta_1)*_{\Phi(e)} \pi_1(\Delta_2)$. In particular, the groups  $\Phi(e), \pi_1(\Delta_1)$ and $\pi_1(\Delta_2)$ are special subgroups of $G_\Gamma$ and the defining graph of $\Phi(e)$ is a complete separating subgraph of $\Gamma$ the vertices of which are finite groups. 
\end{proof}

\section{Proof of Corollary C}

\begin{proof}
The first part of Corollary C is proven in Proposition \ref{finite} and the third part of Corollary C is proven in Proposition B. 

We shall prove parts (ii) and (iv). It follows from Theorem A that  $G_\Gamma$ has more than one end if and only if $\Gamma$ has a complete separating subgraph $\Delta$ the vertex groups of which are finite groups. 
Using the characterization of graph products $G_\Gamma$ with $e(G_\Gamma)=2$, see Proposition B, we obtain the statement of Corollary C(iv).

A characterization of graph products $G_\Gamma$ with $e(G_\Gamma)=1$ is a consequence of Corollary C (i), (iii) and (iv), since a finitely generated group has $0, 1, 2$ or infinite many ends.
\end{proof}

\end{document}